\documentclass[11pt]{amsart}
\usepackage{amssymb}
\usepackage{amsmath,amscd}
\usepackage{combelow}
\usepackage{mathtools}
\usepackage{xcolor}
\usepackage{hyperref}
\usepackage{geometry}
\usepackage{fancyhdr,lipsum}

\newcommand{\arr}{\rightarrow}

\newcommand{\vv}{\vert \vert}
\newcommand{\vvh}{\vert \vert_{\it Hof}}

\newcommand{\R}{\mathbb{R}}
\newcommand{\Q}{\mathbb{Q}}
\newcommand{\N}{\mathbb{N}}
\newcommand{\Z}{\mathbb{Z}}
\newcommand{\C}{\mathbb{C}}
\newcommand{\T}{\mathbb{T}}

\newcommand{\Symp}{\operatorname{Symp}}

\newcommand{\Ham}{\operatorname{Ham}}

\newcommand{\Fix}{\operatorname{Fix}}
\newcommand{\Id}{\operatorname{Id}}
\newcommand{\Li}{\mathcal{L}^*}

\newcommand{\scon}{\overset{C^{\infty}}{\longrightarrow}} 

\newcommand{\larr}{\longrightarrow}

\newtheorem{theorem}{Theorem}
\newtheorem{corollary}[theorem]{Corollary}
\newtheorem{question}[theorem]{Question}
\newtheorem*{question*}{Question}
\newtheorem{definition}[theorem]{Definition}
\newtheorem{lemma}[theorem]{Lemma}
\newtheorem{proposition}[theorem]{Proposition}
\newtheorem{claim}[theorem]{Claim}

\newtheorem*{lemma*}{Lemma}
\newtheorem*{theorem*}{Theorem}
\newtheorem*{remark*}{Remark}
\newtheorem*{definition*}{Definition}
\newtheorem{remark}[theorem]{Remark}
\theoremstyle{remark}

\theoremstyle{remarks}
\newtheorem*{remarks*}{Remarks}
\newtheorem{remarks}[theorem]{Remarks}

\theoremstyle{conjecture}
\newtheorem*{conjecture*}{Conjecture}

\theoremstyle{definition}

\newtheorem*{claim*}{Claim}

\newtheorem*{example*}{Examples}

\title{A H\"older-type inequality for the $C^0$ distance and  Anosov-Katok pseudo-rotations}
\author{Du\v{s}an Joksimovi\'c and Sobhan Seyfaddini}

\begin{document}

\maketitle
\begin{abstract}
    We prove a H\"older-type inequality for Hamiltonian diffeomorphisms relating the $C^0$ norm, the $C^0$ norm of the derivative and the Hofer/spectral norm.  We obtain as a consequence that sufficiently fast convergence in Hofer/spectral metric forces $C^0$ convergence.
    
    The second theme of our paper is the study of pseudo-rotations that arise from the Anosov-Katok method.  As an application of our H\"older-type inequality, we prove a $C^0$ rigidity result for such pseudo-rotations.
\end{abstract}

\tableofcontents

\section{Introduction}
  Denote by $\Ham(M, \omega)$ the group of Hamiltonian diffeomorphisms of a (closed and connected) symplectic manifold $(M, \omega)$. This group possesses two remarkable bi-invariant distances: the Hofer metric and the spectral metric $\gamma$. Exploring the relation between these two metrics and the $C^0$ distance has been a prominent theme in $C^0$ symplectic topology. To this date, research activity has mainly focused on understanding whether the Hofer/spectral metric displays any sort of stability, e.g.~continuity, with respect to the $C^0$ topology. As we will recall in Section \ref{sec:prelim}, much progress has been obtained in this direction over the past decade.

In the opposite direction, it is known that for any two distinct points $x, y \in M$, one can find Hamiltonian diffeomorphisms, arbitrarily close to the identity in Hofer or the spectral distance, which map $x$ to $y$.  Consequently, convergence in the Hofer distance, or the spectral distance, does not imply $C^0$ convergence.  Despite this, we show in this article that sufficiently \emph{fast convergence in Hofer/spectral metric forces $C^0$ convergence}; see Remark \ref{rem:observations}. We achieve this by proving a H\"older-type inequality for Hamiltonian diffeomorphisms relating the $C^0$ norm, the $C^0$ norm of the derivative and the Hofer/spectral norm.

As an application, we prove $C^0$ rigidity results for a certain class of Hamiltonian pseudo-rotations.

\subsection{The inequality}
Let $(M, \omega)$ be a closed and connected symplectic manifold.  We equip $M$ with a Riemannian metric $g$ and denote by $d$ the corresponding distance on $M$. Recall that the $C^0$ distance between two maps $\varphi, \psi$ is given by 

$$d_{C^0}(\varphi, \psi) : = \max_{x \in M} d \left(\varphi(x), \psi(x)\right).$$
In the symplectic literature, there exist inequalities which, under certain assumptions, provide upper bounds for the $\gamma$, or the Hofer norm, in terms of the $C^0$ distance; for example, Hofer's $C^0$--Energy estimate states that for a compactly supported Hamiltonian diffeomorphism of $\R^{2n}$ we have

$$\vv \varphi \vvh \leq C \, \mathrm{diam } (supp(\varphi)) \, d_{C^0}(\varphi, \Id),$$
where $\mathrm{diam } (supp(\varphi))$ is the diameter of the support of $\varphi$. Inequalities of similar nature exist for the spectral norm; see \cite{SeyC0, Shelukhin-18}. In the theorem below we prove an inequality in the opposite direction. Here, $\vv D\varphi \vv$ denotes the $C^0$ norm of the derivative of $\varphi$ as induced by the Riemannian metric.  The Hofer and the spectral norm of $\varphi \in \Ham (M, \omega)$, whose definitions are recalled in Section \ref{sec:prelim}, are denoted by $\vv \varphi \vvh$ and  $\gamma(\varphi)$, respectively.

\begin{theorem} \label{theo:inequality}
Let $(M, \omega)$ be a closed and connected symplectic manifold and $g$ be a Riemannian metric on it. There exists a constant $C := C(M,\omega,g) > 0$ such that for every $\varphi \in \Ham (M,\omega)$ we have
$$d_{C^0}(\varphi, \Id) \leq C \sqrt{\gamma(\varphi)} \; \vv D \varphi \vv.$$
Similarly, for every $\varphi \in \Ham (M,\omega)$ we have $$d_{C^0}(\varphi, \Id) \leq C \sqrt{\vv \varphi \vvh} \; \vv D \varphi \vv.$$
\end{theorem}

\begin{remarks}\label{rem:observations} \normalfont
We make the following observations regarding Theorem \ref{theo:inequality}:
\begin{itemize}
    \item The second inequality is an immediate consequence of the first since $\gamma(\varphi) \leq \vv \varphi \vvh.$
    \item There exists a generalization of the above theorem which holds for Hamiltonian homeomorphisms; see Remark \ref{rem:homeos}.
    \item In the case where $(M, \omega) = (\R^{2n}, \omega_{std})$ and $g$ is the Euclidean metric, our proof of Theorem \ref{theo:inequality} yields the more precise inequality $$d_{C^0}(\varphi, \Id) \leq  \sqrt{\gamma( \varphi) } \; ( 1+ \vv D \varphi \vv),$$
    for every $\varphi \in \Ham_c (\R^{2n}, \omega_{std}).$
    \item As an immediate consequence of Theorem \ref{theo:inequality}, we obtain that sufficiently fast convergence in Hofer/spectral metric forces $C^0$ convergence: suppose that $\varphi_i$ is a sequence of Hamiltonian diffeomorphisms such that $\sqrt{\vv \varphi_i \vv_{Hof}} \hspace{0.7mm} \vv D\varphi_i \vv \to 0$ or $\sqrt{\gamma (\varphi_i)} \hspace{0.7mm} \vv D\varphi_i \vv \to 0.$
    Then,  $$\varphi_i \overset{C^0}{\longrightarrow} \Id.$$
    This is the key observation behind our applications of Theorem \ref{theo:inequality} which are listed in the next section.
\end{itemize}

\end{remarks}

\subsection{Applications}

As an application of the above inequality we prove a relation between different forms of rigidity for Hamiltonian diffeomorphisms.

\begin{definition} \label{def:rigidity}
Let $\nu : \Ham (M,\omega) \arr \R$ be a norm. We say that $\varphi \in \Ham (M,\omega)$ is $\nu$ \emph{rigid} if there exists a sequence of integers $n_i \arr \infty$ such that $\nu (\varphi^{n_i})  \overset{i \arr \infty}{\longrightarrow} 0.$ 

We say that $\varphi$ is \emph{super-exponentially $\nu$ rigid} if for every $c>0$ there exists a sequence $n_i \arr \infty$ such that $\nu( \varphi^{n_i}) < e^{-c n_i}.$ 
\end{definition}

The following is an immediate corollary of Theorem \ref{theo:inequality}.

\begin{corollary} \label{cor:super-gamma-rigid}
Every super-exponentially $\gamma$/Hofer rigid map $\varphi \in \Ham (M,\omega)$ is $C^0$ rigid. 
\end{corollary}

\begin{proof}
We present the proof for the Hofer metric leaving the other case to the reader.
Let $c> \log (\vv D\varphi \vv).$
Since $\varphi$ is super-exponentially Hofer rigid there exists a sequence $n_i \arr \infty$ such that 
\begin{equation} \label{estimate}
    \vv \varphi^{n_i} \vvh < e^{-2c n_i}.
\end{equation}
%
%

By Theorem \ref{theo:inequality} we have
\begin{align*}
    d_{C^0} (\varphi^{n_i}, \Id) &\leq C\sqrt{\vv \varphi^{n_i} \vvh} \; \vv D\varphi^{n_i} \vv \\
                                 &\overset{\eqref{estimate}}{\leq} C \sqrt{e^{-2c n_i} }\hspace{1mm} \vv D\varphi \vv^{n_i} \\
                                 &\leq C e^{n_i (\log (\vv D\varphi \vv) - c)} \overset{i \arr \infty}{\larr} 0.
\end{align*}
This completes the proof of Corollary \ref{cor:super-gamma-rigid}.
\end{proof}

In the next section we present examples of maps which are super-exponentially $\gamma$, or Hofer, rigid.
\subsection{Anosov-Katok pseudo-rotations}
A Hamiltonian diffeomorphism $\varphi$ of a closed and connected symplectic manifold $(M, \omega)$ is called a {\bf pseudo-rotation} if it has finitely many periodic points.\footnote{There exist several working definitions of Hamiltonian pseudo-rotations in the literature; see  \cite[Def.\ 1.1]{CGG19b} and the discussion therein.} Such diffeomorphisms have been of great interest in  dynamical systems and symplectic topology; see, for example, \cite{Anosov-Katok, Fathi-Herman, Fayad-Katok, BCL04, BCL06, Bramham15a, Bramham15b, LeCalvez16, AFLXZ, Ginzburg-Gurel18a,  Ginzburg-Gurel18b, CGG19a, CGG19b, Shelukhin19a, Shelukhin19b}. 

As an example, an irrational rotation of the $2$-sphere is a pseudo-rotation.  Examples with more complicated dynamics, e.g.~with a finite number of ergodic measures, have been produced via the celebrated conjugation method of \textbf{Anosov-Katok}; see \cite{Anosov-Katok, Fayad-Katok, LeRoux-Sey}. The striking feature of the Anosov-Katok method is that it obtains dynamically exotic pseudo-rotations as $C^\infty$ limits of conjugates of ``rotations" (i.e.\ conjugates of elements of a torus action). This motivates the following definition.

\begin{definition}[Anosov-Katok pseudo-rotations] \label{def:AKPR}
Let $(M,\omega)$ be a closed and connected symplectic manifold which admits a Hamiltonian action of a torus $\mathbb{T}^k = \mathbb{R}^k/\mathbb{Z}^k$, i.e.\ a group morphism $R: \mathbb T^k \rightarrow \Ham(M, \omega)$. 

We call a pseudo-rotation $\varphi$ an \emph{Anosov-Katok pseudo-rotation} (abbreviated \emph{AKPR}) if
there exist sequences $h_m \in \Symp (M,\omega)$ and $\alpha_m \in \mathbb{T}^k$ such that
$h_m^{-1} R_{\alpha_m} h_m \scon  \varphi.$
\end{definition}

To the best of our knowledge, all known examples of pseudo-rotations are AKPRs. 

\begin{remark} \label{rmk:conditions} \normalfont 
We remark that the Anosov-Katok method requires the following additional requirements:  the fixed-point set of the toric action $\Fix(R)$ must be finite and the action must be locally free in $M\setminus \Fix(R)$.  Moreover, each symplectomorphism $h_m$ coincides with the identity in a neighborhood $U_k$ of $\Fix(R)$.  Pseudo-rotations obtained this way have the minimal number of fixed points, namely, exactly the fixed points of the toric action. 
We have not incorporated these properties into our definition as they are not needed in the proofs of our results.  
\end{remark}

\subsubsection{Exponentially Liouville AKPRs}
We say that $\alpha \in \mathbb{T}^k = \R^k / \Z^k$ is \emph{exponentially Liouville} if it belongs to the set 
$$\Li :=\left\{ \alpha \in \mathbb{T}^k \hspace{1mm} \vline \hspace{1mm} \forall c>0 \hspace{1mm} \exists k\in \N \hspace{1mm} \text{  s.t. } \hspace{1mm} 0 < \vv k\alpha \vv < e^{-c k} \right\},$$
where $\vv \cdot \vv$ is the distance defined by
\begin{equation} \label{distance}
    \vv (\alpha_1, ... , \alpha_k) \vv := \max_{1 \leq i \leq k} \vert \alpha_i \vert,
\end{equation}
where $\vert \cdot \vert$ is the distance from $0$ on $S^1 = \R / \Z.$

The set $\Li$ forms a residual set (i.e.\ a countable intersection of open and dense sets) in $\mathbb{T}^k$; see \cite[Proposition 5.15]{Ginzburg-Gurel18a} for a proof.

\begin{definition} \label{def:exp_Liouville}
We say that an AKPR $\varphi$ is \emph{exponentially Liouville} if there exist sequences $h_m \in \Symp (M, \omega)$ and $\alpha_m \in \mathbb{T}^k = \R^k / \Z^k$ as in Definition \ref{def:AKPR} such that the sequence $\{\alpha_m\}$ has a limit point which is exponentially Liouville.
\end{definition}

The notion of exponentially Liouville pseudo-rotations dates back to the article \cite{Bramham15b} by Bramham, where it is introduced in the case of the disc.
In \cite{Ginzburg-Gurel18a}, Ginzburg and G\"urel present a generalization of this notion for pseudo-rotations of $\C P^n$. As we explain in Section \ref{sec:C0-rigidity-proof}, under the assumptions from Remark \ref{rmk:conditions}, an AKPR which is exponentially Liouville in the sense of Definition \ref{def:exp_Liouville} is also exponentially Liouville in the sense of Ginzburg and G\"urel, see Remark \ref{rem:exp-liouvill-equiv} below.

\subsubsection{$C^0$ Rigidity.} Our main application of Theorem \ref{theo:inequality} concerns $C^0$ rigidity of pseudo-rotations. 
%
%
Bramham and Ginzburg-G\"urel have shown that exponentially Liouville pseudo-rotations are $C^0$ rigid in the case of the disc \cite{Bramham15b} and $\C P^n$ \cite{Ginzburg-Gurel18a}, respectively.  In Section \ref{sec:proofs}, we will give a new proof of the rigidity results of \cite{Bramham15b, Ginzburg-Gurel18a} based on the inequality from Theorem \ref{theo:inequality} (or more precisely, Corollary \ref{cor:super-gamma-rigid}).  Moreover, we establish $C^0$ rigidity for exponentially Liouville AKPRs on a general closed symplectic manifold.

\begin{theorem}\label{theo:C0-rigidity}
Every exponentially Liouville AKPR is super-exponentially $\gamma$/Hofer rigid. Consequently, every exponentially Liouville AKPR is $C^0$ rigid.
\end{theorem}

This theorem begs the following question.

\begin{question}\label{que:C0-rigidity}
 Is every AKPR $C^0$ rigid?
\end{question}
We refer the reader to an extensive discussion of this and related questions in \cite{AFLXZ}.

\begin{remark*} \normalfont
As we explain below, every AKPR is rigid with respect to $\gamma$ and Hofer norms; see Proposition \ref{prop:hofer-rigidity} and Corollary \ref{cor:gamma-rigidity} below.

It is proven in \cite{Ginzburg-Gurel18a} that every pseudo-rotation of $\C P^n$ is $\gamma$ rigid.\footnote{In \cite{Ginzburg-Gurel18a}, this is proven for pseudo-rotations of $\C P^n$ with exactly $(n+1)$ periodic points. In view of the recent results by Shelukhin \cite{Shelukhin-HoferZehnder}, the proof given in \cite{Ginzburg-Gurel18a} extends to the (hypothetical)  case of pseudo-rotations with more than $(n+1)$ periodic points (which conjecturally do not exist).  We thank Erman \c{C}ineli for this clarification.}  As far as we know, there are no known examples of pseudo-rotations which are not $\gamma$ (or Hofer) rigid. Moreover, $\gamma$-rigid maps are very rare.
For a more detailed discussion on the behaviour of the spectral norm of the iterates of a Hamiltonian diffeomorphism we refer the reader to \cite{CGG22}.
\end{remark*}

\subsection{Entropy}
In this section, we briefly discuss the topological entropy of pseudo-rotations.  We refer the reader to \cite{Hasselblatt-Katok} for a comprehensive introduction to the notion of entropy.  

 One of the main reasons behind the prominence of pseudo-rotations in surface dynamics is that despite having zero topological entropy  they could have interesting dynamics. One way to see that pseudo-rotations of the sphere have zero topological entropy is via Katok's results \cite{Katok-IHES} which imply in particular that a surface diffeomorphism with positive topological entropy has a horseshoe and hence infinitely many periodic points.  This aspect of Katok's theory does not generalize to higher dimensions and so the vanishing of topological entropy for higher dimensional pseudo-rotations is unknown. This question might be easier to address for AKPRs since they are, by definition, $C^\infty$ limits of rotations (which have zero entropy).  However, the question seems to be open even in this simplified setting.
 
 \begin{question}\label{que:entropy}
 Does every AKPR have zero topological entropy?
 \end{question}
 
 Thanks to Theorem \ref{theo:C0-rigidity}, we can give a partial answer to the above question.
 
 \begin{corollary}\label{corol:entropy}
 Every exponentially Liouville $AKPR$ has zero topological entropy.
 \end{corollary}
 \begin{proof}
 It is known that every $C^0$ rigid map has zero topological entropy; see, for example, \cite[Sec.\ 2.4]{AFLXZ}.  
 \end{proof}
 Observe that, by the same reasoning, an affirmative answer to Question \ref{que:C0-rigidity} entails a positive answer to Question \ref{que:entropy}.
 
 \medskip 
 
 Theorems \ref{theo:inequality} \& \ref{theo:C0-rigidity} motivate another interesting question on the growth rate of the derivatives of an AKPR.  
 
  \begin{question}\label{que:growth-rate}
 Let $\varphi$ be an AKPR.  Is it the case that   $$\limsup \frac{\mathrm{log} (\vv D\varphi^n \vv)}{n} = 0?$$
 \end{question}
 
 Once again, in the case of surfaces the answer is known to be positive; this follows from \cite[Theorem C]{AFLXZ}; see also \cite[Sec.\ 2]{Crovisier-notes}.  However, as in the case of Question \ref{que:entropy}, this is open in higher dimensions. Finally, we remark that this question is intimately related to Question \ref{que:entropy} because the topological entropy of any diffeomorphism $\varphi$ is bounded from above by \cite{Ito}   $$ \mathrm{dim}(M)  \limsup \frac{\mathrm{log} (\vv D\varphi^n \vv)}{n}.$$

\subsection*{Acknowledgements} The proof of Theorem \ref{theo:inequality} is inspired by an argument of Avila, Fayad, Le Calvez, Xu and Zhang \cite{AFLXZ}.  We are grateful to Patrice Le Calvez for bringing this argument to our attention.  We warmly thank Viktor Ginzburg and Basak G\"urel for an extensive discussion on the subject of pseudo-rotations as well as for suggesting Corollary \ref{corol:entropy}.  We express our sincere gratitude to Barney Bramham for his careful reading of the first version of the article and for numerous helpful suggestions which have improved our results.  We are also grateful to David Burguet and Sylvain Crovisier for helpful conversations on Questions \ref{que:entropy} \& \ref{que:growth-rate}. Finally, we thank Pierre Berger, Erman \c{C}ineli, Fr\'ed\'eric Le Roux, Abror Pirnapasov and Egor Shelukhin for helpful comments and conversations.  

Both authors are supported by the ERC Starting Grant number 851701.


\section{Preliminaries}\label{sec:prelim}
In this section we introduce some of our notation and recall the necessary background from symplectic geometry.

Throughout the section $(M, \omega)$  denotes a closed and connected symplectic manifold.  Recall that a symplectomorphism is a diffeomorphism $\varphi: M \to M$ such that $\varphi^* \omega = \omega$.  We denote by $\Symp(M, \omega)$ the set of all symplectomorphisms of $(M, \omega)$.  We let $\Symp_0(M, \omega)$ denote those elements of  $\Symp(M, \omega)$ which are isotopic to the identity.

  Hamiltonian diffeomorphisms constitute an important class of examples of symplectic diffeomorphisms which are defined as follows. A smooth Hamiltonian $H \in C^{\infty} ([0,1] \times M)$  gives rise to a time-dependent vector field $X_H$ which is defined via the equation: $\omega(X_H(t), \cdot) = -dH_t$.  The Hamiltonian flow of $H$, denoted by  $\phi^t_H$, is by definition the isotopy generated by $X_H$.  A  Hamiltonian diffeomorphism is a diffeomorphism which arises as the time-one map of a Hamiltonian flow.  The set of all Hamiltonian diffeomorphisms is denoted by $\Ham(M, \omega)$; it forms a (normal) subgroup of $\Symp_0(M, \omega)$.

\subsection{The Hofer metric}
The {\bf Hofer norm} $\vv \varphi \vvh$ of any $\varphi \in \Ham(M,\omega)$ is defined as follows.  First, to a Hamiltonian $G \in C^{\infty}([0,1] \times M)$, we associate the (pseudo)norm
\[ \vv G \vv_{1,\infty} := \int_0^1 \left( \max_{M}(G_t) - \min_{M}(G_t) \right) dt.\]
We then define
\[ \vv \varphi \vvh := \inf \lbrace \vv G \vv_{1,\infty} \hspace{1mm} \vert \hspace{1.25mm} \varphi = \phi^1_G \rbrace.\]
This quantity is invariant under conjugation, i.e.~$\vv \psi^{-1} \varphi \psi  \vvh = \vv \varphi \vvh$ because  $\phi^t_{H\circ\psi} = \psi^{-1} \phi^t_H \psi$; see \cite[Sec. 5.1, Prop. 1]{hofer-zehnder}, for example. 

We can define a metric on $\Ham(M,\omega)$, the {\bf Hofer metric}, by
\[ d_{\it Hof}(\varphi,\psi) = \vv \varphi^{-1} \circ \psi \vvh .\]
As mentioned above, this yields a non-degenerate, bi-invariant metric known as the Hofer metric.  Non-degeneracy, which was established by Hofer for $\R^{2n}$ \cite{Hofer-metric}, by Polterovich for rational symplectic manifolds \cite{Polterovich93}, and by Lalonde-McDuff in full generality \cite{Lalonde-McDuff}, is a consequence of the celebrated energy-capacity inequality:   suppose that $U \subset M$ is symplectomorphic to the closed Euclidean ball of radius $r$. If $\varphi \in \Ham (M,\omega)$ displaces $U$, that is $\varphi(U) \cap U = \emptyset$, then
\begin{equation}\label{eq:energy-capacity-hofer}
    \pi r^2  \leq  \vv \varphi \vvh.
\end{equation}

The above version of the energy-capacity inequality is due to Usher \cite{usher-sharp}.

\subsection{The spectral metric }
The group $\Ham(M, \omega)$ admits another conjugation invariant norm known as the {\bf spectral norm}.  It is usually denoted by $\gamma :\Ham(M, \omega) \rightarrow \R$ and it was constructed by Viterbo \cite{viterbo92} on $\R^{2n}$, via generating function theory, by Schwarz \cite{schwarz} on aspherical $(M, \omega)$ and by Oh \cite{oh} on general $(M, \omega)$ using Hamiltonian Floer theory.  

Being a conjugation invariant norm means that $\gamma$ satisfies the following properties:

\begin{enumerate}
    \item  $\gamma(\varphi) \geq 0$ with equality only if $\varphi = \Id$,
    \item $\gamma(\varphi) = \gamma(\varphi^{-1})$,
    \item  $\gamma(\varphi \psi) \leq \gamma(\varphi) + \gamma(\psi)$,
    \item  $\gamma(\psi^{-1}\varphi \psi) = \gamma(\varphi)$. 
\end{enumerate}
As before, we can define a metric on $\Ham(M,\omega)$, the {\bf spectral metric}, by
\[ d_\gamma(\varphi,\psi) = \gamma (\varphi^{-1} \circ \psi) .\]
As in the case of the Hofer metric, non-degeneracy (i.e.~the first item on the above list) is a consequence of the energy-capacity inequality for $\gamma$:  suppose that $U \subset M$ is symplectomorphic to the closed Euclidean ball of radius $r$. If $\varphi \in \Ham (M, \omega)$ displaces $U$, that is $\varphi(U) \cap U = \emptyset$,  then

\begin{equation}\label{eq:energy-capacity-gamma}
    \pi r^2  \leq  \gamma(\varphi).
\end{equation}

The spectral norm is not easy to compute except in very specific cases. For example, if $H$ is a sufficiently $C^2$-small autonomous Hamiltonian then by \cite[Proposition 4.1]{usher-sharp} we have that %
\begin{equation}\label{eq:gamma-small}
    \gamma(\phi^1_H) = \max_M H - \min_M H.
\end{equation}

\subsubsection*{Comparison to Hofer's metric}
It is known that the spectral metric is smaller than Hofer's metric : $d_\gamma \leq d_{Hof}$.  

However, in certain cases the two metrics coincide:  if $H$ is a sufficiently $C^2$-small autonomous Hamiltonian, then it follows from Equation \eqref{eq:gamma-small} that 

\begin{equation}\label{eq:gamma=Hofer}
    \gamma(\phi^1_H) = \vv \phi^1_H \vvh = \max_M H - \min_M H.
\end{equation}

\medskip

In this article we will only rely on the above properties of the spectral metric and so we will not recall its definition which is quite involved and relies on the machinery of filtered Floer homology and the theory of spectral invariants.

\subsection{Relation between the Hofer/spectral metric and $d_{C^0}$}
 
We recall here some facts about the behaviour of the Hofer, or the spectral, metric with respect to the $C^0$ topology. 

  As far as we know, there exists no example of a compact symplectic manifold on which the Hofer metric is $C^0$ continuous.  This can easily be seen in the case of the unit ball in $\R^{2n}$ by considering a sequence of Hamiltonian diffeomorphisms whose supports shrink to a point (hence $C^0$ convergence to the $\Id$) and whose Calabi invariant, which bound the Hofer norm from below, are bounded away from zero.    Despite this, the Hofer metric still displays some interesting properties with respect to the $C^0$ topology \cite{EPP, Sey_descent, BHS, Kawamoto} and, as Le Roux suggests \cite{Leroux-six}, it might even be lower semi-continuous in the $C^0$ topology.

As for the spectral metric $\gamma$, it is known to be $C^0$ continuous on a large class of manifolds: this was proven for $\R^{2n}$ by Viterbo \cite{viterbo92}, closed surfaces by the second author \cite{SeyC0}, closed aspherical manifolds by Buhovsky-Humilière-Seyfaddini \cite{BHS}, $\C P^n$ by Shelukhin \cite{Shelukhin-18} and negative monotone manifolds by Kawamoto \cite{Kawamoto}.  It is expected that these results extend to arbitrary symplectic manifolds. 

In the opposite direction, given two distinct points $x,y$, it is fairly easy to construct $\varphi_n \in \Ham (M,\omega)$ such that $\varphi_n(x) = y$ and $d_{\it Hof}(\varphi_n, \Id ) \to 0$.  Clearly, since the Hofer norm is larger than $\gamma$, we also have $\gamma(\varphi_n) \to 0$. So we see that $d_{C^0}$ is not continuous with respect to either of the Hofer or $\gamma$ metrics.  Observe that, as a consequence of our Theorem \ref{theo:inequality}, we can conclude that $\vv D\varphi_n \vv \to \infty $.

\section{Proof of the inequality}\label{sec:proof_inequality}

We present here the proof of Theorem \ref{theo:inequality}.  The heart of the argument is a local version of the theorem which we state separately as a claim.

\begin{claim}\label{cl:inequality_local}
Let $(M, \omega)$ be a closed and connected symplectic manifold and $g$ be a Riemannian metric on it. There exist constants $C := C(M,\omega,g) > 0$ and $\delta := \delta(M,\omega,g)>0$ such that the following holds. For every $\varphi \in \Ham (M,\omega)$ such that $\gamma(\varphi) < \delta$ we have
$$d_{C^0}(\varphi, \Id) \leq C \sqrt{\gamma(\varphi)} \; \vv D \varphi \vv.$$
\end{claim}

Let us explain why the above claim implies the theorem.  Since the manifold $M$ is compact its diameter, $\displaystyle \mathrm{diam(M)}:= \sup_{x,y \in M} d(x,y)$ is bounded.  Without loss of generality, we may suppose that the constant $C$ from the above claim satisfies the inequality
\begin{equation}\label{eqn:bound_C}
   C \geq  \frac{\mathrm{diam}(M)}{\sqrt{\delta}}.
\end{equation}

Claim \ref{cl:inequality_local} proves Theorem \ref{theo:inequality} in the case where $\gamma(\varphi) < \delta$. To complete the proof, suppose that  $\gamma(\varphi) \geq \delta$.  Then, note that by \eqref{eqn:bound_C}

$$ C \sqrt{\gamma(\varphi)} \; \vv D \varphi \vv \geq \mathrm{diam}(M) \; \vv D \varphi \vv \geq  \mathrm{diam}(M), $$ 
where the right-most inequality follows from the fact that, with $\varphi$ being symplectic, we have $\vv D \varphi \vv \geq 1$.
Clearly, $\mathrm{diam}(M)$ is an upper bound for $d_{C^0}$ and so we conclude from the above inequality that
$$d_{C^0}(\varphi, \Id) \leq C \sqrt{\gamma(\varphi)} \; \vv D \varphi \vv.$$

It remains to prove Claim \ref{cl:inequality_local}.
\begin{proof}[Proof of Claim \ref{cl:inequality_local}]
Let $x \in M.$ Choose a Darboux chart $\psi_x: U_x \arr \R^{2n}$ about $x.$ 
Choose compact neighbourhoods $K_x$ and $K_{x}'$ of $x$ such that 
$$K_x \subseteq \operatorname{int} (K_x') \subseteq K_x' \subseteq U_x$$
Then $\{\operatorname{int} (K_x) \}_{x \in M}$ covers $M$ and since $M$ is compact we can choose a finite subcover $\{\operatorname{int} (K_i)\}_{1 \leq i \leq k}.$
Define 
\begin{align}
    &\varepsilon := \min_{1 \leq i \leq k} d(\partial K_i, \partial K_i') \label{eps} \\
    & L: = \max_{1 \leq i \leq k} \vv D \psi_i^{-1} \vert_{\psi_i(K_i')} \vv, \label{L}
\end{align}
where $\vv \cdot \vv$ is the operator norm with respect to $g$ and the Euclidean metric.
Roughly, one can understand the constant $L$ as a uniform Lipschitz constant for the maps $\psi_i^{-1} \vert_{\psi_i(K_i')}$ with respect to the Euclidean distance and the metric $d$ (induced by $g$).  

Define 
\begin{equation}\label{delta}
    \delta := \frac{\pi \varepsilon^2}{4L^2}, \quad C := \frac{8L}{\sqrt{\pi}}. 
\end{equation}
We prove that $C$ and $\delta$ satisfy the statement of Theorem \ref{theo:inequality}. 
For that let $\varphi \in \Ham (M,\omega)$ such that $\gamma (\varphi) < \delta.$ 
Let $x \in M$ and $i$ such that $x \in \operatorname{int} (K_i).$ Let $B_{r}(\psi_i(x)) \subseteq \R^{2n}$ be the ball of radius 
$r := 2\sqrt{\frac{\gamma(\varphi)}{\pi}} < \frac{\varepsilon}{L}$ about $\psi_i(x).$ 
Since $x \in K_i$ from \eqref{eps},\eqref{L} it follows that $B_{r}(\psi_i(x))$ lies inside the Darboux chart $\psi_i(K_i').$ Consider the set  $\psi^{-1}_i (B_{r}(\psi_i(x))) $; for any point $p \in \psi^{-1}_i (B_{r}(\psi_i(x))) $ we have $d(x,p) \leq Lr =  2L\sqrt{\frac{\gamma(\varphi)}{\pi}},$ and so 
\begin{equation*}
    \psi^{-1}_i (B_{r}(\psi_i(x))) \subset B_x,
\end{equation*}
where $B_x \subset M$ denotes the ball of radius $2L\sqrt{\frac{\gamma(\varphi)}{\pi}}$ centered at $x$ (w.r.t. metric $d$).

From the above containment and the energy-capacity inequality \eqref{eq:energy-capacity-gamma} we conclude that $\varphi$ cannot displace $B_x,$ i.e.
$$\varphi (B_x) \cap B_x \neq \emptyset.$$
Choose $y = \varphi(z) \in \varphi (B_x) \cap B_x.$
Then, we have 
\begin{align*}
    d(x, \varphi (x)) & \leq d(x,y) + d(y, \varphi(x)) \\
                      & = d(x,y) + d(\varphi(z), \varphi(x)) \\
                      & \leq \mathrm{diam}(B_x) + \sup_{x \in M} \; \mathrm{diam}(\varphi(B_x)),
                      \end{align*}
where $\mathrm{diam}$ denotes the diameter of the set.                      
We conclude that 

\begin{equation}\label{eq:C0-estimate}
d_{C^0}(\varphi, \Id) \leq 4L\sqrt{\frac{\gamma(\varphi)}{\pi}} + \sup_{x \in M} \; \mathrm{diam}(\varphi(B_x)).
\end{equation}

Since $\varphi$ is differentiable, we have  $\mathrm{diam}(\varphi(B_x)) \leq \vv D\varphi \vv \, \mathrm{diam}(B_x) \leq \vv D\varphi \vv \cdot 4L\sqrt{\frac{\gamma(\varphi)}{\pi}}$.  Hence, we obtain
\begin{align*}
                      d_{C^0}(\varphi, \Id) 
                      & \leq  4L\sqrt{\frac{\gamma(\varphi)}{\pi}}(1+ \vv D\varphi \vv) \\
                      & \leq C\sqrt{\gamma(\varphi)} \hspace{0.7mm} \vv D\varphi \vv,
\end{align*}
where the last inequality follows from the fact that $D\varphi$ is a symplectic map and hence $\vv D\varphi \vv \geq 1$ (since it has at least one eigenvalue of norm greater than or equal to $1$).

This completes the proof of Theorem \ref{theo:inequality}.
\end{proof}

\begin{remark} \label{rem:homeos} \normalfont
In the above proof, the differentiability of the map $\varphi$ is used only in the very last step.  The inequality \eqref{eq:C0-estimate} holds for Hamiltonian homeomorphisms, assuming that $\gamma$ can be defined for such homeomorphisms (which is the case for a large class of symplectic manifolds; see \cite{SeyC0, BHS, Shelukhin-18, Kawamoto}.)
\end{remark}

\section{Proofs of the rigidity results} \label{sec:proofs}

Throughout this section we suppose that $(M, \omega)$  is a closed and connected symplectic manifold of dimension $2n$ admitting an AKPR $\varphi$ as in Definition \ref{def:AKPR}.  Recall that this means that $(M, \omega)$ admits a Hamiltonian torus action $R: \mathbb T^k \rightarrow \Ham(M, \omega)$ and, moreover,  there exist sequences $h_m \in \Symp (M,\omega)$ and $\alpha_m \in \mathbb{T}^k$ such that
$h_m^{-1} R_{\alpha_m} h_m \scon  \varphi.$

Before proceeding to prove our results, we introduce some notation.
For any fixed $ \alpha \in \mathbb{T}^k= \R^k / \Z^k$, we refer to the map 
$R_{\alpha}$ as \emph{the rotation in the direction} $\alpha.$   Each of the circle factors in the torus  $\mathbb{T}^k= \R^k / \Z^k$ yields a Hamiltonian circle action,
or equivalently a periodic Hamiltonian flow with period 1; let $H_1, \ldots, H_k \in C^\infty(M)$ be (autonomous) 
 Hamiltonians generating these circle actions.  Denote by   $$\mu := (H_1, ..., H_k): M \arr \R^k$$ the corresponding momentum map of the action, and define 
\begin{equation*} 
    \vv \mu \vv_{\infty} := \max_{1 \leq i \leq k} \vv H_i \vv_{\infty}. 
\end{equation*}
%
%



\medskip

We begin with the following lemmas which will be used repeatedly in the proofs of main results.

\begin{lemma}\label{lem:hofer-equality}
Let $\alpha$ be any accumulation point of the sequence $\{\alpha_m\}.$ Then, for every integer $j$, we have $\vv \varphi^j \vvh = \vv R_{j\alpha}\vvh$.
\end{lemma}

\begin{proof}[Proof of Lemma \ref{lem:hofer-equality}]
After possibly passing to a subsequence we may assume that $\alpha_m \arr\alpha$.  For every integer $j$, we have 

\begin{equation}\label{rot-Hor-conv0}
    \vv h_m^{-1} R^j_{\alpha_m}  h_m\vvh = \vv R_{\alpha_m}^j \vvh \larr \vv R_{\alpha}^j \vvh = \vv R_{j\alpha} \vvh.
\end{equation}
On the other hand, since $h_m^{-1} R_{\alpha_m}  h_m \scon \varphi$ it follows that
$$h_m^{-1} R_{\alpha_m}^j  h_m = (h_m^{-1} R_{\alpha_m}  h_m)^j \scon \varphi^j,$$ and therefore
\begin{equation} \label{pr-Hof-conv0}
\vv h_m^{-1} R^j_{\alpha_m}  h_m \vvh \larr \vv \varphi^j \vvh.
\end{equation}
The lemma follows immediately from \eqref{rot-Hor-conv0} and \eqref{pr-Hof-conv0}.
\end{proof}


\begin{lemma}\label{lem:hofer-inequality}
Let $\alpha \in \mathbb{T}^k = \R^k / \Z^k.$ Then, the following holds
\begin{equation} \label{hof_estimate}
    \vv R_{\alpha} \vvh \leq  2k\vv \alpha \vv \cdot \vv \mu \vv_{ \infty}, 
\end{equation}
where $\vv \cdot \vv$ is as in \eqref{distance}.
\end{lemma}

\begin{proof}
Let $\alpha = (\alpha_1,..., \alpha_k) \in \mathbb{T}^k,$ and denote by $\mu = (H_1,...,H_k) :  M \arr \R^k$ the momentum map.
Notice that the rotation in the direction $(0,...,\alpha_i,0,...,0)$ is generated by the Hamiltonian $\pm \vert \alpha_i \vert H_i,$ for every $1 \leq i \leq k$; recall that $\vert \cdot \vert$ is the distance from $0$. 
Hence, the Hamiltonian 
\begin{equation} \label{hamiltonian}
    H_{\alpha}  := \sum_{i=1}^k \pm \vert \alpha_i \vert \cdot H_i 
\end{equation} 
generates the rotation $R_{\alpha}.$
Therefore, 
$$\vv R_{\alpha} \vvh \leq \vv H_{\alpha} \vv_{(1,\infty)}  \leq \sum_{i=1}^k \vert \alpha_i \vert \cdot \vv H_i \vv_{(1,\infty)} \leq 2k \vv \alpha \vv \cdot \vv \mu \vv_{\infty}.$$
This completes the proof of Lemma \ref{lem:hofer-inequality}.
\end{proof}

\subsection{Further remarks on the limit points of the sequence $\{\alpha_m\}$. }
Let $\varphi$ be an AKPR and $h_m \in \Symp (M, \omega),$ $\alpha_m \in \T^k$ sequences such that
$h_m^{-1} R_{\alpha_m} h_m \scon  \varphi.$ The main goal of this section is to establish the following result concerning the limit points of the sequence $\{\alpha_m\}$.  The contents of this section are not used elsewhere in the paper.

\begin{proposition}\label{prop:rot_num}
Let $\alpha \in T^k = \R^k / \Z^k$ be a limit point of the sequence $\{\alpha_m\}$.  Then, $\alpha$ has at least one irrational coordinate. 
Moreover, in the case when $k=1,$ if $\alpha'$ is another limit point of the same sequence, then one 
of $\{\alpha - \alpha', \alpha + \alpha'\}$ is rational. 
\end{proposition}

\begin{remark*}[Rotation number of AKPRs arising from an $S^1$ action] \normalfont
When $k=1,$ Proposition \ref{prop:rot_num} above allows us to make sense of the rotation number of an AKPR which turns out to be well-defined as an element of $\R/\Q$, up to a sign. 
Moreover, it follows from the proposition that (in the case when $k=1$) if one of the limit points of the sequence $\{\alpha_m\}$ is exponentially Liouville then all of them are (for any choice of sequences $h_m$ and $\alpha_m$). This is a consequence of the fact that being exponentially Liouville is invariant under rational translations and multiplication by integers.
\end{remark*}

\begin{proof}[Proof of Proposition \ref{prop:rot_num}]
Let $\alpha \in \mathbb{T}^k = \R^k / \Z^k$ be an accumulation point of the sequence $\{\alpha_m\}.$ 
As a consequence of Lemma \ref{lem:hofer-equality}, we obtain that $\alpha$ has at least one irrational coordinate. Namely, assume to the contrary that $\alpha \in \Q^k / \Z^k.$ Then there exists $m \in \N$ such that $m\alpha = 0 \in \mathbb{T}^k$ and thus  $R_{\alpha}^m = R_{m\alpha} = \Id$.  From the lemma, we conclude that $\vv \varphi^m \vvh = 0$ and so $\varphi^m = \Id$, but this is not possible since $\varphi$ is a pseudo-rotation. Hence $\alpha$ has at least one irrational coordinate, i.e.~$\alpha \notin \Q^k / \Z^k.$

Now we prove the second part. For that we assume that $k=1$ (i.e. that we have a circle action). Assume that $\alpha$ and $\alpha'$ are two distinct accumulation points of $\{\alpha_m\}.$
Choose $\varepsilon>0$ such that $\varepsilon H$ is sufficiently $C^2$-small so that, by \eqref{eq:gamma=Hofer}, for every $\vert \eta \vert \leq \varepsilon,$ it holds that 

\begin{equation} \label{eq:Hof-small}
    \vv R_{\eta}  \vvh = \vert \eta \vert (\max_M H - \min_M H).
\end{equation}

Choose $l \in \N$ such that both $\vert \{l\alpha\} \vert$ and $\vert \{l\alpha'\} \vert$ \footnote{Here $\{x\}$ denotes the fractional part of $x \in \R.$ It is defined as $\{x\} := x - [x],$ where $[x]$ is the largest integer which is not bigger than $x.$} are smaller than $\varepsilon,$ where $\vert \cdot \vert$ denotes the distance from $\Z.$
Then \eqref{eq:Hof-small} implies that
\begin{align*}
    \vv R_{\alpha}^l \vvh &= \vv R_{\{l\alpha\}} \vvh = \vert \{l \alpha\} \vert   \cdot (\max_M H - \min_M H), \\
    \vv R_{\alpha'}^l \vvh &= \vv R_{\{l\alpha'\}} \vvh = \vert \{l \alpha'\} \vert \cdot  (\max_M H - \min_M H).
\end{align*}
Now by Lemma \ref{lem:hofer-equality}, we get that 
\begin{equation*} 
    \vv R_{\alpha}^l \vvh = \vv R_{\alpha'}^l \vvh.
\end{equation*}
and hence
$$\vert \{l \alpha\} \vert = \vert \{l\alpha'\} \vert.$$
%
Therefore one of the following holds $$\{l\alpha\} = \{l\alpha'\} \hspace{2mm} \text{  or  }  \hspace{2mm}  \{l\alpha\} = 1-\{l\alpha'\}.$$

Consider the first case when $\{l\alpha\} = \{l\alpha'\}.$ Then from the fact that
\begin{align*}
    l\alpha &= [l\alpha] + \{l\alpha\} \\
    l\alpha' &= [l \alpha'] + \{l\alpha'\}
\end{align*}
it follows that
$$\alpha-\alpha' = \frac{[l\alpha] - [l \alpha']}{l} \in \Q.$$
In the second case using the same argument we get that $\alpha+\alpha' \in \Q.$
This completes the proof of Proposition \ref{prop:rot_num}.
\end{proof}

\subsection{$C^0$ rigidity of exponentially Liouville pseudo-rotations}\label{sec:C0-rigidity-proof}
In this section we prove our results on $C^0$ rigidity of pseudo-rotations; these are our main applications of Theorem \ref{theo:inequality}.
\subsubsection{The case of AKPRs}
We present here the proof of Theorem \ref{theo:C0-rigidity}.

\begin{proof}[Proof of Theorem \ref{theo:C0-rigidity}]
Let $\varphi \in \Ham (M,\omega)$ be an exponentially Liouville AKPR and let $c>0$. Choose $h_m \in \Symp (M,\omega), \alpha_m \in \mathbb{T}^k$ such that $\{\alpha_m\}$ has a limit point $\alpha$ which is exponentially Liouville.
%
%
Then $\alpha \in \Li$ implies that there exists a sequence $n_i \arr \infty$ such that 
\begin{equation} \label{estimate1}
    \vv n_i \alpha \vv < e^{-c n_i},
\end{equation}
where  $\vv \cdot \vv$ is as in \eqref{distance}.

By Lemma \ref{lem:hofer-equality} and Lemma \ref{lem:hofer-inequality} we have
\begin{equation}\label{exp-convergence}
    \vv \varphi^{n_i} \vvh = \vv R_{n_i \alpha} \vvh  \leq 2k\vv n_i \alpha \vv \cdot \vv \mu \vv_{ \infty}.
\end{equation}
Now from \eqref{estimate1} we get that 
$$\vv \varphi^{n_i} \vvh \leq e^{-c n_i} \cdot 2k \vv \mu \vv_{\infty},$$
and hence $\varphi$ is super-exponentially Hofer rigid. 
Then, from Corollary \ref{cor:super-gamma-rigid} we conclude that $\varphi$ is also $C^0$ rigid.
This completes the proof of Theorem \ref{theo:C0-rigidity}.
\end{proof}

\subsubsection{The case of $\mathbb{C}P^n$}
In the case when $M = {\C}P^n$ Ginzburg and G\"urel proved $C^0$ rigidity for an (a priori) more general class of pseudo-rotations, see \cite[Theorem 5.16]{Ginzburg-Gurel18a}; the same result for $\mathbb{C}P^1$ is due to Bramham \cite{Bramham15b}. 
We will explain how one can use the inequality from Theorem \ref{theo:inequality} to simplify the proof of the $C^0$ rigidity result, using similar ideas as in the proof of Theorem \ref{theo:C0-rigidity} above.

We first briefly recall Ginzburg and G\"urel's definition of \emph{exponentially Liouville} pseudo-rotations.  Given a fixed point $x$ of a Hamiltonian diffeomorphism, using the choice of a generating Hamiltonian $F$ such that $\varphi = \phi^1_F,$ and a capping $\overline{x}$ one can associate to $\overline{x}$ a real number $\Delta_F(\overline{x})$ called the mean index of $x$. 
Now, given a pseudo-rotation\footnote{ Ginzburg and G\"urel define pseudo-rotations of $\C P^n$ to be those Hamiltonian diffeomorphisms which have exactly $n+1$ periodic points.} $\varphi$ of $\mathbb{C}P^n$, its mean index vector $\vec{\Delta}_F(\varphi)$ is the vector in $\mathbb{T}^{n+1} = \R^{n+1}/2(n+1)\Z^{n+1}$ whose components record the mean indices of the fixed points of $\varphi$; see \cite{Ginzburg-Gurel18a} for the definition of the mean index and other details.  
According to Ginzburg and G\"urel, a pseudo-rotation $\varphi$ is exponentially Liouville if its mean index vector is exponentially Liouville, i.e.  $\forall c>0, \hspace{1mm} \exists k\in \N $ such that  $ \vv k \vec{\Delta}_F(\varphi) \vv < e^{-c k},$
where $\vv \cdot \vv$ is the distance from 0. 
Being exponentially Liouville is independent of the choice of the Hamiltonian $F$.

\medskip

Now, it turns out that for any pseudo-rotation $\varphi$ we have  $$\gamma(\varphi^k) \leq c \vv k \vec{\Delta}_F(\varphi) \vv$$
where $c$ is a constant depending only $ \vec{\Delta}_F(\varphi)$; see \cite[Remark 5.5]{Ginzburg-Gurel18a}. Hence, if $\vec{\Delta}_F(\varphi)$ is exponentially Liouville, we can pick the sequence $\{n_i\}_{i \in \N}$ such that 
$$\sqrt{\gamma (\varphi^{n_i})} \hspace{0.7mm} \vv D\varphi^{n_i} \vv \to 0.$$
The $C^0$ rigidity for such $\varphi$ then follows immediately from Theorem \ref{theo:inequality}.

\medskip

We remark here that, under assumptions from Remark \ref{rmk:conditions} (which hold in all known examples of PRs constructed using the method of Anosov and Katok), an AKPR being  exponentially Liouville in our sense is equivalent to that of Ginzburg and G\"urel.
Here, we briefly outline the proof of this statement.
%
We denote by $\overline{x}_1, ... ,\overline{x}_{l}$ the fixed point of the $\mathbb{T}^k$-action with the trivial cappings,\footnote{Note that in the setting of \cite{Ginzburg-Gurel18a} $l=n+1,$ but we do not impose this as a condition since we do not need it for the proof of the statement, see Remark \ref{rem:exp-liouvill-equiv}.} 
and denote by 
$$\mu = (H_1, ... , H_k) : M \arr \R^k$$ the momentum map of the action.
Then under assumptions from Remark \ref{rmk:conditions},
using standard properties of the mean index (see e.g. \cite[Section 2.4]{GG1}) one can show the following.

\begin{itemize}

    \item[(i)] For every $1 \leq i \leq l, 1 \leq j \leq k,$ and $\xi \in \R$ it holds that $${\Delta}_{\xi H_j} (\overline{x}_i) \in \xi \Z.$$ 
    Hence, using the fact that the rotation in the direction $\alpha =(\alpha_1,...,\alpha_k) \in \mathbb{T}^k := \R^k / \Z^k$ is generated by a Hamiltonian 
    $$H_{\alpha} := \sum_{j=1}^k \alpha_j H_j,$$
    using properties of the mean index ((MI2), (MI7), (MI8) from \cite[Section 2.4]{GG1}) we get that 
    $$\Delta_{H_{\alpha}} (\overline{x}_i) \in \alpha \cdot \Z^k := \{n_1 \alpha_1+ ... + n_k\alpha_k  \hspace{1mm} \vert \hspace{1mm} (n_1,...,n_k) \in \Z^k\}.$$

    
    \item[(ii)] Let $\varphi$ be an AKPR and let $\alpha_m \in \mathbb{T}^k$ be a sequence associated to $\varphi$ as in Definition \ref{def:AKPR}.
    Let $\alpha$ be an accumulation point of the sequence $\alpha_m.$
    Then for every Hamiltonian $F$ which generates $\varphi$ and every $1 \leq i \leq l$ it holds that $${\Delta}_F(\overline{x}_i) - \Delta_{ H_{\alpha}}(\overline{x}_i)  \in \Z.$$

\end{itemize}

Set $$\vec{\Delta}_{H_{\alpha}}(R_{\alpha}) := (\Delta_{H_{\alpha}}(\overline{x}_1), ..., \Delta_{H_{\alpha}}(\overline{x}_{l})).$$ 
Now note that (ii) implies that $\vec{\Delta}_{F}(\varphi)$ is exponentially Liouville iff $\vec{\Delta}_{H_{\alpha}}(R_{\alpha})$ is exponentially Liouville,\footnote{This follows from the fact that the property of being an exponentially Liouville vector is preserved by a translation by a vector with rational coordinates.}
while it can be checked that (i) implies that
$\vec{\Delta}_{H_{\alpha}}(R_{\alpha})$ is exponentially Liouville iff $\alpha \in \Li.$ 
Combining these two items we conclude that $\vec{\Delta}(\varphi)$ is exponentially Liouville if and only if $\alpha \in \Li,$ which proves the equivalence of Definition \ref{def:exp_Liouville} with the definition of exponentially Liouville PRs of Ginzburg and G\"urel.

\begin{remark} \label{rem:exp-liouvill-equiv} \normalfont
The same argument works for any closed symplectic manifold that admits a Hamiltonian $\mathbb{T}^k$-action and an AKPR, which satisfy properties stated in Remark \ref{rmk:conditions}. To be more precise, we say that a PR $\varphi$ is \emph{exponentially Liouville in the sense of Ginzburg and G\"urel} if its mean index vector is exponentially Liouville.\footnote{Note that this definition does not require any additional assumptions on symplectic manifold besides that it admits PRs (not even the existence of a Hamiltonian $\T^k$-action). We impose the extra assumptions only in the case of AKPRs.} Then, an AKPR (from such a $\T^k$-action) is exponentially Liouville in the sense of Ginzburg and G\"urel if and only if it is exponentially Liouville in the sense of Definition \ref{def:exp_Liouville}.

\end{remark}

\subsection{Rigidity with respect to the Hofer and spectral metrics}\label{sec:-rigidity-Hofer-spectral}

In this section we prove another form of rigidity of AKPRs. Namely, that with respect to the Hofer metric, see Definition \ref{def:rigidity}.

\begin{proposition}[Hofer rigidity of AKPRs] \label{prop:hofer-rigidity}
Every AKPR is Hofer rigid.
\end{proposition}


As an immediate consequence we have the following.

\begin{corollary}[$\gamma$-rigidity of AKPRs] \label{cor:gamma-rigidity}
Every AKPR is $\gamma$ rigid. 
\end{corollary}

\begin{proof}[Proof of Proposition \ref{prop:hofer-rigidity}]
Let $\alpha$ be an accumulation point of the sequence $\alpha_m.$

Consider the sequence $\{n \alpha \} \subset \mathbb T^k$; this sequence accumulates at $0$, i.e.\ there exists a sequence of integers $n_i \arr \infty$ and $\vv n_i \alpha \vv \arr 0.$ (Note that, by Proposition \ref{prop:rot_num}, $\alpha$ has an irrational coordinate, and hence $n \alpha \neq 0$ for all $n$.)
Then by Lemma \ref{lem:hofer-equality} and Lemma \ref{lem:hofer-inequality} we get
$$\vv \varphi^{n_i} \vvh = \vv R_{n_i \alpha } \vvh \leq 2k\vv n_i \alpha \vv \cdot \vv \mu \vv_{\infty} \overset{i \arr \infty}{\larr}  0,$$
where $\mu$ is the momentum map.
This completes the proof of Proposition \ref{prop:hofer-rigidity}.
\end{proof}


\subsection{Lagrangian Poincare recurrence} \label{sec:Lag-Poin-recurrence}

The following conjecture was stated independently by Ginzburg and Viterbo around 2010.

\begin{conjecture*}[Lagrangian Poincar\'e Recurrence]
For every compactly supported Hamiltonian diffeomorphism $\varphi$ of a symplectic manifold $(M, \omega)$
and every closed Lagrangian submanifold $L \subseteq M$ there exists a sequence of iterations $n_i \arr \infty$ such that $\varphi^{n_i}(L) \cap L \neq \emptyset.$
Moreover, the density of the sequence $n_i$ is related to a symplectic capacity of $L.$
\end{conjecture*}

Ginzburg and G\"urel proved the above conjecture in the case where $\varphi$ is a pseudo-rotation of $\C P^n$ and $L$ any closed Lagrangian whose certain homological capacity is positive; the proof relies on $\gamma$ rigidity of $\varphi$, see \cite[Theorem 5.8]{Ginzburg-Gurel18a}.  Using similar ideas, we prove the above conjecture for AKPRs.

\begin{theorem} \label{theo:recurrence}
Let $(M, \omega)$ be a closed (and connected) symplectic manifold admitting an AKPR $\varphi$. Then for every subset $A \subseteq M$ with positive displacement energy $e(A) >0$ there exists a sequence $n_i \arr \infty$ such that $\varphi^{n_i}(A) \cap A \neq \emptyset.$ Moreover, there exist $C>0$ and an integer $1 \leq d \leq k,$ both depending only on $\varphi,$ such that 
\begin{equation} \label{eq:density-intersections}
    \liminf_{i \arr \infty} \frac{\# \{ j \leq i \hspace{1mm} \vert \hspace{1mm} \varphi^j(A) \cap A \neq \emptyset \}}{i} \geq \min\{1, C e(A)^d\}.
\end{equation}
\end{theorem}

\begin{proof}[Proof of Theorem \ref{theo:recurrence}]

 Let $h_m \in \Symp  (M,\omega), \alpha_m \in \mathbb{T}^k$ be sequences associated to $\varphi$ as in Definition \ref{def:AKPR}. Let $\alpha$ be an accumulation point of the sequence $\alpha_m.$  
By Lemma \ref{lem:hofer-equality} and Lemma \ref{lem:hofer-inequality}, $\vv \varphi^{j} \vvh = \vv R_{j\alpha} \vvh \leq 2k \vv j \alpha \vv \cdot \vv \mu \vv_{\infty}$. Define $C'': = \frac{1}{2k\vv \mu \vv_{ \infty}}$. Observe that if   $\vv j \alpha\vv < C'' \cdot e(A)$, then $\varphi^j(A) \cap A \neq \emptyset$.

Denote by $G \subseteq \mathbb{T}^k$ be the subgroup topologically generated by $\alpha.$ Then $G$ is the closure of the positive semi-orbit $\{ j \alpha \hspace{1mm} \vert \hspace{1mm} j \in \N\},$ and without loss of generality we may assume that $G$ is isomorphic to a torus.\footnote{In general, $G$ is isomorphic to the direct product of a cyclic group of order $k_0$ and a torus. Thus replacing $\varphi$ with $\varphi^{k_0}$ we may assume that $G$ is a torus.} 
Set $d = \dim G.$ 
Then for the volume of the intersection of the $\varepsilon$-neighborhood $B(\varepsilon)$ of 0 in $\mathbb{T}^k$ with $G$ we have that 
\begin{equation*}
   \operatorname{vol} (B(\varepsilon) \cap G) \geq C' \varepsilon^d,
\end{equation*}
where $C'>0$ depends only on $\alpha$ (and hence depends it only on $\varphi$).
Since the sequence $\{j \alpha\}_{j \in \N}$ is equidistributed in $G \subseteq \mathbb{T}^k,$ we get 
$$\lim_{i \arr \infty} \frac{\# \{ j \leq i \hspace{1mm} \vert \hspace{1mm} \vv j \alpha \vv <  C'' \cdot e(A)\}}{i} \geq  \min\{1, C' (C'' \cdot e(A))^d\}. $$
This then implies \eqref{eq:density-intersections} (for $C:=C' {C''}^d)$ and completes the proof of Theorem \ref{theo:recurrence}.
\end{proof}

As an immediate corollary of Theorem \ref{theo:recurrence} and Chekanov's theorem \cite{Chekanov} we obtain the following.

\begin{corollary} [Lagrangian Poincar\'e Recurrence for AKPRs] \label{theo:Lagr-recurrence}
Let $(M, \omega)$ be a closed (and connected) symplectic manifold that admits a Hamiltonian $\mathbb{T}^k$ action and let $\varphi \in \Ham (M,\omega)$ be an AKPR. Then for every closed Lagrangian submanifold $L \subseteq M$ there exists a sequence $n_i \arr \infty$ such that $\varphi^{n_i}(L) \cap L \neq \emptyset.$ Moreover, there exist $C>0$ and an integer $1 \leq d \leq k,$ both depending only on $\varphi,$ such that 
$$\liminf_{i \arr \infty} \frac{\# \{ j \leq i \hspace{1mm} \vert \hspace{1mm} \varphi^j(L) \cap L \neq \emptyset \}}{i} \geq \min \{1, C \cdot e(L)^d \}.$$
\end{corollary}


 \bibliographystyle{alpha}
 \bibliography{biblio}

{\small

\medskip
 \noindent Du\v{s}an Joksimovi\'c\\
\noindent Universit\'e Paris-Saclay, Orsay, France\\
 {\it e-mail:} dusan.joksimovic@universite-paris-saclay.fr

\medskip
 \noindent Sobhan Seyfaddini\\
\noindent Sorbonne Universit\'e and Universit\'e de Paris, CNRS, IMJ-PRG, F-75006 Paris, France.\\
 {\it e-mail:} sobhan.seyfaddini@imj-prg.fr
 
}

\end{document}